\documentclass{amsart}	

\usepackage{fullpage}

\usepackage{amssymb, amsmath, amsthm, color, amsfonts, caption}

\newtheorem{theorem}{Theorem}
\numberwithin{theorem}{section}
\numberwithin{equation}{section}
\newtheorem{corollary}[theorem]{Corollary}
\newtheorem{lemma}[theorem]{Lemma}
\newtheorem{definition}[theorem]{Definition}
\newtheorem{proposition}[theorem]{Proposition}
\newtheorem{remark}[theorem]{Remark}

\newcommand{\G}{\Gamma}

\newcommand{\C}{\mathbb{C}}
\newcommand{\R}{\mathbb{R}}
\newcommand{\Z}{\mathbb{Z}}
\newcommand{\N}{\mathbb{N}}
\renewcommand{\H}{\mathbb{H}}

\newcommand{\IM}{\rm{Im}}

\DeclareMathOperator{\impart}{Im}
\DeclareMathOperator{\lcm}{lcm}
\DeclareMathOperator{\SL}{SL}
\DeclareMathOperator{\sgn}{sgn}

\newcommand{\smmat}[4]{\left(\begin{smallmatrix}#1 & #2\\ #3 & #4\end{smallmatrix}\right)}

\allowdisplaybreaks

\begin{document}

\author{Allison Arnold-Roksandich, Brian Diaz,  Erin Ellefsen, Holly Swisher}

\title{Extending a catalog of mock and quantum modular forms to an infinite class}


\thanks{This work was partially supported by the NSF REU Site Grant DMS-1359173}

\begin{abstract} Utilizing a classification due to Lemke Oliver of eta-quotients which are also theta functions (here called eta-theta functions), Folsom, Garthwaite, Kang, Treneer, and the fourth author constructed a catalog of mock modular forms $V_{mn}$ having weight $3/2$ eta-theta function shadows and showed that these mock modular forms when viewed on certain sets of rationals, transform as quantum modular forms under the action of explicit subgroups.  In this paper, we introduce an infinite class of functions that generalizes one row of the catalog, namely the $V_{m1}$, and show that the functions in this infinite class are both mock modular and quantum modular forms.
\end{abstract}

\maketitle

\section{Introduction and Statement of Results} \label{intro}

Dedekind's $\eta$-function, which is defined by
\begin{equation}\label{def:eta}
\eta(\tau):= q^{\frac{1}{24}}\prod_{n\geq 1} (1-q^n),
\end{equation}
where $\tau$ is in the upper half-plane $\H=\{\tau\in \C \mid \IM(\tau)\}$, and $q=e^{2\pi i \tau}$ is one of the most well-known half-integral weight modular forms.  Quotients of the form 
\[
\prod_{j=1}^{c} \eta(a_j\tau)^{b_j},
\]
for $b_j \in \Z$ are called \emph{eta-quotients}, and have been of interest in many areas of mathematics, including combinatorics, number theory, and representation theory (see \cite{Ono-Web}, \cite{Andrews}, \cite{CN}, for example).   

Classifications of eta-quotients have also been of interest.  Dummit, Kisilevsky, and McKay \cite{DKM} classified all multiplicative \emph{eta-products} (restricting to positive $b_j$), and later Martin \cite{Martin} classified multiplicative integer weight eta-quotients.  Furthering these classifications, Lemke Oliver \cite{LO} recently classified all eta-quotients which are also \emph{theta functions}, which are of the form
\[
\theta_\chi(\tau) := \sum_n \chi(n) n^{\nu} q^{n^2},
\] 
where $\chi$ is an even (resp. odd) Dirichlet character, $\nu$ is $0$ (resp. $1$), and the sum above is over $n\in \Z$ when $\chi$ is trivial, or over $n \in \N$ when $\chi$ is not trivial.  These functions are known to be ordinary modular forms of weight $1/2 + \nu$.

We refer to a function which is both an eta-quotient and a theta function as an \emph{eta-theta} function.  In Lemke Oliver's classification there are six odd (character) eta-theta functions $E_m$, and eighteen even (character) eta-theta functions $e_n$ (some of which are twists by certain principal characters).  

Modular theta functions also arise in the theory of harmonic Maass forms and mock modular forms via the shadow operator.  A \emph{harmonic Maass form} $\widehat{f}$, as originally defined by Bruinier and Funke \cite{BF}, naturally decompose into two parts as $\widehat{f} = f^+  + f^-$, where $f^+$ is called the \emph{holomorphic part} of $\widehat{f}$, and $f^-$ is called the \emph{non-holomorphic part} of $\widehat{f}$. All of Ramanujan's original mock theta functions turn out to be examples of holomorphic parts of harmonic Maass forms \cite{Zwegers2}, and we now define after Zagier \cite{ZagierB} a \emph{mock modular form} to be any such holomorphic part of a harmonic Maass form (for which the non-holomorphic part is not zero).  Mock modular forms come naturally equipped with a \emph{shadow}, which is a certain modular cusp form obtained via a differential operator defined in Section \ref{prelim}. 

In recent work of Folsom, Garthwaite, Kang, Treneer, and the fourth author (Folsom, et al.) \cite[Theorem 1.1]{FGKST} a catalog of functions $V_{mn}$ are constructed, utilizing the even eta-theta functions $e_n$ classified by Lemke Oliver, where each $V_{mn}$ is a weight $1/2$ mock modular form with respect to a particular congruence subgroup $A_{mn}$ such that the shadow of $V_{mn}$ is directly related to the odd eta-theta functions $E_m$ from Lemke Oliver's classification.

Moreover, it is additionally shown in \cite[Theorem 1.2]{FGKST} that the mock modular forms $V_{mn}$ are quantum modular forms.  A \emph{quantum modular form}, as defined by Zagier \cite{Zagier}, and slightly generalized in the relevant book by Bringmann, Folsom, Ono, and Rolen \cite{BFOR}, is a complex function defined on an appropriate subset of rational numbers (instead of the upper half-plane), which transforms like a modular form up to the addition of an error function that is suitably continuous or analytic in $\mathbb R$.  In the emerging theory of quantum modular forms it is interesting to ask how quantum modular forms may arise from mock modular forms (see \cite{BFR2, BOPR, FOR, BFOR}, for example).

In this paper we introduce an infinite class of functions $V_\alpha$ that generalize one row, namely the $V_{m1}$, of the catalog from Folsom et al. in \cite{FGKST}.  We show that these functions $V_\alpha$ are both mock modular and quantum modular forms.  In order to state our results we need some notation of Zwegers \cite{zwegers}.  Throughout the paper, we also use the notation $e(a) := e^{2\pi ia}$ and $\zeta_a := e^\frac{2\pi i}{a}$.

\begin{definition}\label{def:mu}
For $\tau \in \mathbb{H}$ and $u,v \in \C\backslash (\Z\tau + \Z)$,
%\begin{equation}\label{series}
\[
\mu(u,v;\tau) := \frac{e^{\pi i u}}{\vartheta(v;\tau)}\sum_{n\in\Z}\frac{(-1)^n e^{2\pi i n v}q^{\frac{n(n+1)}{2}}}{1 - e^{2\pi i u}q^{n}},
\]
%\end{equation}
where, for $z \in \C$ and $\tau \in \mathbb{H}$,
$$ \vartheta(z;\tau) := \sum_{v \in \Z +\frac{1}{2}} e^{\pi i v^2 \tau+2\pi i v\left(z+\frac{1}{2}\right)}.$$
\end{definition}

We can now define our class of functions.

\begin{definition}\label{deffunction}
Let $\alpha = \frac{A}{2C}\tau + \frac{a}{4}$, such that $A,a \in \Z$, $C\in \N$, $0 \leq a \leq 3$, $\gcd(A,C)=1$, and $2\alpha\not\in \Z + \tau \Z$. We then define $V_\alpha(\tau)$ for $\tau \in \H$ by
\begin{align*}\label{function}
V_\alpha(\tau) &:= i^{a+1}q^{-\frac{(2A-C)^2}{8C^2}} \mu\left(2\alpha, \frac{\tau}{2}; \tau\right).
\end{align*}
\end{definition}

\begin{remark}\label{rmk_extension}
Considering the catalog of functions $V_{mn}$ from Folsom et al. \cite{FGKST}, we see that by choosing $(A,C,a)$ from the set
\[
\{(1,4,1), (1,4,0), (1,3,1), (1,12,0), (5,12,0), (1,6,1), (1,3,0)\},
\]
$V_\alpha$ yields the functions $V_{11}, V_{21}, V_{31}, V_{4'1}, V_{4''1}, V_{51}, V_{61}$.  In particular, we note that we can generate the first row of each table in \cite{FGKST} using the $V_\alpha$ function.  We further note that there are infinitely many appropriate choices for $A$ and $C$.
\end{remark}

We first show that the functions $V_\alpha(\tau)$ are mock modular forms for appropriate groups $A_\alpha$ defined in Section \ref{mock}.

\begin{theorem}\label{mockmodularity}
For $\alpha = \frac{A}{2C}\tau + \frac{a}{4}$ where $a \in \{0,1,2,3\}$, $0 < \frac{A}{C} < 1$, and $\frac{A}{C} \neq \frac12$ when $a$ is odd, the functions $V_{\alpha}$ are mock modular forms of weight $1/2$ with respect to the congruence subgroups $A_{\alpha}$. \end{theorem}

In addition, we prove quantum modularity results for $V_\alpha(\tau)$ for appropriate subsets $S_\alpha$ of rationals, and appropriate groups $G_\alpha\subseteq A_\alpha$, which are defined in Section \ref{quantum}.  In particular, we prove the following theorem.

\begin{theorem}\label{quantummodularity}
For $\alpha = \frac{A}{2C}\tau + \frac{a}{4}$ where $a \in \{0,1,2,3\}$, $0 < \frac{A}{C} < 1$, and $\frac{A}{C} \neq \frac12$ when $a$ is odd, the functions $V_{\alpha}$ are quantum modular forms on the sets $S_{\alpha}$ with respect to the groups $G_{\alpha}$. In particular, the following are true. 
\begin{enumerate}

\item Suppose $a \in \{0,2\}$. Then for all  $\tau \in \mathbb{H} \cup S_{\alpha} \backslash \{ -1\}$, we have that 
$$V_{\alpha}(\tau) - \zeta_8^{-1}(\tau + 1)^{-1/2}V_{\alpha}\left(\frac{\tau}{\tau + 1}\right) = \frac{-i}{2}e\left(\frac{-A}{2C}\right)\int_1^{i\infty}\frac{g_{A/C,1/2}(z)}{\sqrt{-i(z + \tau)}}dz.$$

\item Suppose $a \in \{1,3\}$. Then for all $\tau \in \mathbb{H} \cup S_{\alpha} \backslash \{ -\frac{1}{2} \}$, we have that 
$$V_{\alpha}(\tau) - (2\tau+1)^{-1/2}V_{\alpha}\left(\frac{\tau}{2\tau + 1}\right) = \frac{-i}{2}\int_{1/2}^{i\infty}\frac{g_{A/C,0}(z)}{\sqrt{-i(z+\tau)}}dz.$$

\item For all $\tau \in \mathbb{H}\cup S_{\alpha}$, we have if $C$ is even,
$$V(\tau) - (-1)^{A + \frac{C}{2}}e\left(\frac{C}{8}\right)e\left(\frac{C(2A-C)^2}{8C^2}\right)V(\tau + C) = 0,$$
and if C is odd,
$$V(\tau) + e\left(\frac{C}{4}\right)e\left(\frac{(2A-C)^2}{4C}\right)V(\tau + 2C) = 0.$$
\end{enumerate}
\end{theorem}

In Section \ref{prelim}, we will provide necessary background. In Section 3, we prove Theorem \ref{mockmodularity}, and in Section 4, we prove Theorem \ref{quantummodularity}. 

\section{Preliminaries}\label{prelim}

It is well-known that the theta functions $\vartheta(z;\tau)$ defined in Definition \ref{def:mu} may be written as  
\begin{equation}\label{eq:theta}
\vartheta(v;\tau) = -iq^{\frac{1}{8}}e^{-\pi i v}\prod_{n\geq 1}(1-q^n)(1-e^{2\pi i v}q^{n-1})(1-e^{-2\pi i v}q^n).
\end{equation}

Zwegers \cite{zwegers} showed that the $\mu$ function defined in Definition \ref{def:mu}, has useful transformation properties.  To state these, recall that for $u\in\C$ and $\tau\in\mathbb{H}$, Zwegers defines the Mordell integral $h$ by

\begin{equation}\label{Mordell}
h(u)=h(u;\tau):=\int_{\R}\frac{e^{\pi i\tau x^2 - 2\pi ux}}{\cosh{\pi x}}dx.
\end{equation}

\begin{lemma}[Zwegers \cite{zwegers}]  \label{lem_mu} 
Let $\mu(u,v):=\mu(u,v;\tau)$ and $h(u;\tau)$ be defined as in Definition \ref{def:mu} and (\ref{Mordell}). Then we have
\begin{enumerate}
\item[(1)] $\mu(u+1,v)=-\mu(u,v)$,
\item[(2)] $\mu(u,v+1)=-\mu(u,v)$,
\item[(3)] $\mu(-u,-v)=\mu(u,v)$,
\end{enumerate}
as well as the modular transformation properties,
\noindent
\begin{enumerate}
\item[(4)] $\mu(u,v;\tau+1)=e^{-\frac{\pi i}{4}}\mu(u,v;\tau)$,
\item[(5)] $\frac{1}{\sqrt{-i\tau}}e^{\pi i(u-v)^2/\tau}\mu\left(\frac{u}{\tau},\frac{v}{\tau};-\frac{1}{\tau}\right)+\mu(u,v;\tau)=\frac{1}{2i}h(u-v;\tau)$.
\end{enumerate}
\end{lemma}

\noindent The function $h$ also has some useful properties. 

\begin{lemma}[Zwegers \cite{zwegers}] \label{htrans}
The function h has the following properties:
\begin{enumerate}
\item[(1)] $h(z) + h(z+1) = \frac{2}{\sqrt{-i\tau}}e^{\pi i\left(z + \frac{1}{2}\right)^2/\tau}$
\item[(2)] $h(z) + e^{-2\pi iz - \pi i\tau}h(z + \tau) = 2e^{-\pi iz - \pi i\tau/4}$
\item[(3)] $h$ is an even function of $z$.
\end{enumerate}
\end{lemma}

%%%%%

We now recall some basic facts about harmonic Maass forms. The definition we state below is essentially that of Bruinier and Funke in \cite{BF}. Note that for $k\in\frac{1}{2}\Z$, the weight $k$ hyperbolic Laplacian operator $\Delta_{k}$ is defined as
$$
\Delta_k = -y^2\left(\frac{\partial^2}{\partial x^2}+\frac{\partial^2}{\partial y^2} \right)
+iky\left(\frac{\partial}{\partial x}+i\frac{\partial}{\partial y}\right),
$$
where $\tau=x+iy$.

\begin{definition}\label{MaassForm}
Given $k\in \frac{1}{2}\Z$, a finite index subgroup $\Gamma\subset\SL_2(\Z)$, and 
a multiplier $\psi$ (so $\psi:\Gamma\rightarrow\mathbb{C}$ with $|\psi|=1$),
a \emph{harmonic Maass form} of weight $k$ for the subgroup $\Gamma$, with multiplier 
$\psi$, is any smooth function $\widehat{f}:\H\rightarrow\mathbb{C}$
such that the following three properties hold.
\begin{enumerate}
\item
For all $\gamma=\smmat{a}{b}{c}{d} \in \Gamma$ and 
$\tau \in \H$, $$\widehat{f}(\gamma \tau) = \psi(\gamma)(c\tau+d)^{k} \widehat{f}(\tau).$$ 
\item
We have that $\Delta_k(\widehat{f})=0$. 
\item
The function $\widehat{f}$ has linear exponential growth at the cusps in the 
following sense. For each 
$\gamma=\smmat{a}{b}{c}{d}\in\SL_2(\Z)$, there exists a polynomial 
$p_{A}(\tau)\in \mathbb{C}[q^{-\frac{1}{N}}]$ (with $N$ a positive integer) such that 
$(c\tau+d)^{-k}\widehat{f}(\tau) - p_{A}(\tau) = O(e^{-\epsilon y})$ as $y\rightarrow \infty$ 
for some $\epsilon > 0$. 
\end{enumerate}
\end{definition}

The Fourier series of a harmonic Maass form $\widehat{f}$ of weight $k$ naturally 
decomposes as the sum of a holomorphic part, $f^+$, and non-holomorphic part, $f^-$, such that \[\widehat{f} = f^+ + f^-.\]  Using terminology of Zagier \cite{Zagier}, the holomorphic part $f^+$ is called a \emph{mock modular form} of weight $k$.  

Recall the function $\mu(u,v;\tau)$ defined in Definition \ref{def:mu}. Zwegers \cite{zwegers} showed that $\mu$ can be completed to a non-holomorphic modular Jacobi form, by
\begin{equation}\label{muhat}
\widehat{\mu}(u,v;\tau) := \mu(u,v;\tau)+\frac{1}{2}R(u-v;\tau),
\end{equation}
where 
\[
R(u;\tau) := \sum_{v\in\frac{1}{2}+\Z}\left(\sgn(v) - 2\int_0^{\left(v+\frac{\impart(u)}{\impart(\tau)}\right)\sqrt{2\impart(\tau)}}e^{-\pi t^2}\,dt \right)(-1)^{v-\frac{1}{2}}e^{\pi iv^2\tau-2\pi ivu}.
\]

The following lemma demonstrates how $\widehat{\mu}$ transforms.  Here we let $\psi$ denote the character in the modular transformation for $\eta(\tau)$, which can be defined (as in Knopp \cite[Ch. 4, Thm. 2]{Knopp}) for $\gamma = \smmat{a}{b}{c}{d} \in \SL_2(\Z)$ by
\begin{align}\label{psidef}
\psi(\gamma) &=
	\left\{
	\begin{array}{ll}
		\big(\frac{d}{|c|} \big)
		\exp\left(\frac{\pi i}{12}\left(
			(a+d)c - b d(c^2-1) - 3c		
		\right)\right)
		&
		\mbox{ if } c \equiv 1 \pmod{2},
		\\				
		\left(\frac{c}{d}\right)
		\exp\left(\frac{\pi i}{12}\left(
			(a+d)c - b d(c^2-1) + 3d - 3 - 3cd		
		\right)\right)
		&
		\mbox{ if } d\equiv 1\pmod{2},
	\end{array}
	\right.
\end{align}
where $\left(\frac{\cdot}{\cdot}\right)$ is the Kronecker symbol.

\begin{lemma}\label{zwemutrans}\cite[Prop. 1.11(1,2)]{zwegers}
Let $\widehat{\mu}(u,v;\tau)$ be defined as in (\ref{muhat}). Then \begin{enumerate}
\item $\widehat{\mu}(u+k\tau+\ell,v+m\tau+n;\tau) = (-1)^{k+\ell+m+n}e^{\pi i(k-m)^2\tau+2\pi i(k-m)(u-v)}\widehat{\mu}(u,v;\tau)$, for $k,\ell,m,n\in\Z$.
\item $\widehat{\mu}\left(\frac{u}{c\tau+d},\frac{v}{c\tau+d};\frac{a\tau+b}{c\tau+d}\right) = \psi(\gamma)^{-3}(c\tau+d)^{\frac{1}{2}}e^{-\pi ic(u-v)^2/(c\tau+d)}\widehat{\mu}(u,v;\tau)$, for $\gamma = \smmat{a}{b}{c}{d}\in SL_2(\Z)$.
\end{enumerate}
\end{lemma}

As discussed in work of Jennings-Shaffer and the fourth author \cite{JSS} (see (2.19) and following discussion), for $a\in\mathbb{Q}$,
\begin{align}\label{eqn:growth_check}
R(a\tau-b) &= p(\tau) + O(y^{-\frac{1}{2}}e^{-\pi a^2y - \epsilon y}),
\end{align}
as $y\rightarrow\infty$, where $\epsilon>0$ and $p(\tau)$ is some 
rational function of a fractional power of $q$. As such,
for $u_1,u_2,v_1,v_2\in\mathbb{Q}$, it follows that
\begin{equation}\label{hmfgrowth}
q^{-\frac{(u_1-v_1)^2}{2}}\widehat{\mu}\left( u_1\tau+u_2,v_1\tau+v_2  ;\tau\right)
\end{equation}
meets the prescribed growth conditions of a harmonic Maass form.

Following the notation of Folsom et. al. in \cite{FGKST}, we make the following definition.  Let $x_\tau$ be a complex function of $\tau$. Given $\gamma = \smmat{a}{b}{c}{d}\in\SL_2(\Z)$, define 
\begin{equation}\label{tildedef}
\tilde{x}_{\gamma,\tau} := x_{\gamma\tau}\cdot (c\tau +d).
\end{equation}

The following lemma follows directly from Lemma \ref{zwemutrans}, and is given in Folsom et. al \cite{FGKST}. 

\begin{lemma}\label{muhattrans}\cite[Lemma 3.1]{FGKST}
Let $\gamma = \smmat{a}{b}{c}{d}\in SL_2(\Z)$, $\tau\in\H$, and $u_\tau,v_\tau\in \C\setminus(\Z\tau+\Z)$. Suppose $\tilde{u}_{\gamma,\tau} = u_\tau + k_\gamma\cdot\tau+\ell_\gamma$, and $\tilde{v}_{\gamma,\tau} = v_\tau + r_\gamma\cdot\tau+s_\gamma$, for some integers $k_\gamma,\ell_\gamma,r_\gamma,s_\gamma$. Then, we have that \begin{multline*}
\widehat{\mu}(u_{\gamma\tau},v_{\gamma\tau};\gamma\tau) = \widehat{\mu}\left(\frac{\tilde{u}_{\gamma,\tau}}{c\tau+d},\frac{\tilde{v}_{\gamma,\tau}}{c\tau+d};\gamma\tau\right) 
=\psi(\gamma)^{-3}(-1)^{k_\gamma+\ell_\gamma+r_\gamma+s_\gamma}(c\tau+d)^{\frac{1}{2}}q^{\frac{(k_\gamma-r_\gamma)^2}{2}}\\
\cdot e\left(\frac{-c(\tilde{u}_{\gamma,\tau} - \tilde{v}_{\gamma,\tau})^2}{2(c\tau+d)}+(k_\gamma-r_\gamma)(u_\tau-v_\tau)\right) \widehat{\mu}(u_\tau,v_\tau;\tau).
\end{multline*}
\end{lemma}

In order to use Lemma \ref{muhattrans} for our functions $V_\alpha(\tau)$, we use the following lemma and its corollary, which follow directly from the definition of $\tilde{x}_{\gamma,\tau}$ and are also stated in Folsom et. al \cite{FGKST}.

\begin{lemma}\label{mockmodcong}\cite[Lemma 3.3]{FGKST}
Let $x_\tau\in\C\setminus(\Z\tau +\Z)$ be of the form \[x_\tau = \frac{\alpha\tau +\beta}{N}\] where $N\in\N$, $N\nmid \alpha$, and $N\nmid\beta$. For fixed $\gamma = \smmat{a}{b}{c}{d} \in SL_2(\Z)$, we have that $\tilde{x}_{\gamma,\tau} - x_\tau \in \Z\tau + \Z$ if and only if the following congruences hold: \[\alpha a + \beta c \equiv \alpha\pmod{N}\]\[\alpha b + \beta d\equiv \beta\pmod{N}.\]
\end{lemma}

The following corollary considers the case when $N$ divides one of $\alpha$ and $\beta$.

\begin{corollary}\label{congcoro}\cite[Corollary 3.4]{FGKST}
Let $x_\tau\in\C\setminus(\Z\tau +\Z)$ be of the form \[x_\tau = \frac{\alpha\tau +\beta}{N}\] where $N\in\N$. When $N|\alpha$ and $\beta$ is relatively prime to $N$, then $\tilde{x}_{\gamma,\tau} - x_\tau \in \Z\tau + \Z$ if and only if $\gamma\in \Gamma_1(N)$. Similarly, if $N|\beta$ and $\alpha$ is relatively prime to $N$, then $\tilde{x}_{\gamma,\tau} - x_\tau \in \Z\tau + \Z$ if and only if $\gamma\in \Gamma^1(N)$.
\end{corollary}

Here we note that we use the notation $\G^0(N)$ and $\G^1(N)$ to refer to the  corresponding ``lower triangular" versions of the congruence subgroups $\G_0(N)$ and $\G_1(N)$, respectively.

An additional lemma from Folsom et. al. \cite{FGKST} that we use is given below.  Let $\tau\in \H$, $a,b\in \R$, and $u,v \in \C \backslash (\Z\tau + \Z)$.  Then define  
\begin{equation}\label{Mabdef}\widehat{M}_{a,b}(\tau) := -\sqrt{2} e^{2\pi ia\left( b+\frac{1}{2}\right)} q^{-\frac{a^2}{2}}\widehat{\mu}(u,v;\tau).
\end{equation}

\begin{lemma}\label{LaplaceLem}\cite[Prop. 2.9]{FGKST}
Let $\tau\in\H$, and $u,v\in\C\setminus(\Z\tau +\Z)$. If $u-v = a\tau - b$ for some $a,b\in\R$, then the function $\widehat{M}_{a,b}(\tau)$ satisfies \[\Delta_{\frac{1}{2}}(\widehat{M}_{a,b}(\tau)) = 0.\]
\end{lemma}

\noindent We also will require the following definition and several results of Zwegers \cite{zwegers}.

\begin{definition}
Let $a,b \in \R$ and $\tau \in \mathbb{H}$; then
$$g_{a,b}(\tau) := \sum_{v \in a+ \Z} ve^{\pi i v^2 \tau+2\pi i vb}.$$
\end{definition}

The following lemma gives transformation formulas for the functions $g_{a,b}(\tau)$, and show in particular that $g_{a,b}(\tau)$ is a modular form of weight $3/2$ when $a,b$ are rational.

\begin{lemma}[Zwegers \cite{zwegers}]\label{lem_gabprop} The function $g_{a,b}$ satisfies the following:
\begin{enumerate}
\item[(1)] $g_{a+1,b}(\tau)=g_{a,b}(\tau)$,
\item[(2)] $g_{a,b+1}(\tau)=e^{2\pi ia}g_{a,b}(\tau)$,
\item[(3)] $g_{-a,-b}(\tau)=-g_{a,b}(\tau)$,
\item[(4)] $g_{a,b}(\tau +1)=e^{-\pi ia(a+1)}g_{a,a+b+\frac12}(\tau)$,
\item[(5)] $g_{a,b}(-\frac{1}{\tau})=ie^{2\pi iab}(-i\tau)^{3/2}g_{b,-a}(\tau)$.
\end{enumerate}
\end{lemma}
The functions $g_{a,b}(\tau)$ are related to $R$ and $h$ in the following theorem.

\begin{theorem}[Zwegers, Thm. 1.16 of \cite{zwegers}]\label{thm_z116}
For $\tau \in \mathbb{H}$, we have the following two results.

\begin{enumerate}
\item
When $a\in (-\frac12, \frac12)$ and $b \in \R$,
\[
\int_{-\overline{\tau}}^{i\infty} \frac{g_{a+\frac{1}{2},b+\frac{1}{2}}(z)}{\sqrt{-i(z+\tau)}}dz = -e^{2\pi i a(b+\frac{1}{2})}q^{-\frac{a^2}{2}} R(a\tau-b;\tau).
\]
\item
Also, when $a,b \in (-\frac12, \frac12)$,
\[
\int_{0}^{i\infty} \frac{g_{a+\frac{1}{2},b+\frac{1}{2}}(z)}{\sqrt{-i(z+\tau)}}dz =-e^{2\pi i a(b+\frac{1}{2})}q^{-\frac{a^2}{2}} h(a\tau-b;\tau).
\]
\end{enumerate}
\end{theorem}

We also use the following extension of Theorem \ref{thm_z116} given in Folsom et. al \cite{FGKST}.

\begin{lemma}[Lemma 2.8 of \cite{FGKST}]\label{lem_Zlem}  
Let $\tau \in \mathbb H$.  
\begin{enumerate}
\item 
For $b\in\mathbb R\setminus \frac12\mathbb Z$, $$\int_{-\overline{\tau}}^{i\infty} \frac{g_{1,b+\frac12}(z)}{\sqrt{-i(z+\tau)}} dz = -ie\left(-\frac{\tau}{8} + \frac{b}{2}  \right)R\left(\frac{\tau}{2}-b;\tau\right) + i.$$
\item
For $b\in (-\frac12,\frac12)\setminus\{0\}$, $$\int_{0}^{i\infty} \frac{g_{1,b+\frac12}(z)}{\sqrt{-i(z+\tau)}}dz = -ie\left(-\frac{\tau}{8} + \frac{b}{2}  \right)h\left(\frac{\tau}{2}-b;\tau\right) + i.$$
\item 
For $a\in (-\frac12,\frac12)\setminus\{0\}$, $$\int_{0}^{i\infty} \frac{g_{a+\frac12,1}(z)}{\sqrt{-i(z+\tau)}}dz = -e\left(-\frac{a^2}{2}\tau + a  \right)h\left(a\tau-\frac12;\tau\right) + \frac{e(a)}{\sqrt{-i\tau}}.$$
\end{enumerate}
\end{lemma}

We will also use of the following theorem of Kang \cite{Kang}\footnote{The notation used here is slightly different than Kang's, and is taken from Zwegers \cite{zwegers}.}.  

\begin{theorem}[Kang \cite{Kang}] \label{kangthm}
If $\alpha\in\mathbb{C}$ such that $\alpha \not \in \frac12 \mathbb Z \tau + \frac12 \mathbb Z$, then
\[
\mu\left(2\alpha, \frac{\tau}{2};\tau\right) = iq^\frac18 g_2(e^{2\pi i \alpha};q^\frac12) -e^{-2\pi i \alpha}q^\frac18 \frac{\eta(\tau)^4}{\eta(\frac{\tau}{2})^2 \vartheta(2\alpha;\tau)},
\]
where $g_2$ is the universal mock theta function defined by
\[
g_2(z;q):=\sum_{n=0}^\infty \frac{(-q)_n q^{n(n+1)/2}}{(z;q)_{n+1}(z^{-1}q;q)_{n+1}}.
\]
\end{theorem}  

The following lemma will be useful for determining when $V_\alpha$ is well defined on subsets of rationals.  Let
\begin{equation}\label{S_def}
S = \left\{\frac{h}{k} \in \mathbb{Q}\, \bigg|\,  h \in \mathbb{Z}, k \in \mathbb{N}, \gcd(h,k)=1, h \equiv 1 \!\!\!\!\! \pmod2 \right\}.
\end{equation}

\begin{lemma}\label{lem:eta_vanishing}
Fix $A,a \in \Z$, $C\in \N$, with $0 \leq a \leq 3$, such that $\frac{A}{2C}$ and $\frac{a}{4}$ are not both in $\frac12 \mathbb{Z}$.  Suppose $\widetilde{S}\subseteq S\subseteq\mathbb{Q}$ is a set of rationals such that for all $n \geq 1$ and all $\frac{h}{k} \in \widetilde{S}$, 
\[
\frac{nh}{k} \pm \left(\frac{Ah}{Ck} + \frac{a}{2} \right) \not\in \mathbb{Z}.
\]
Then
\[
\mu\left(\frac{Ah}{Ck} + \frac{a}{2}, \frac{h}{2k}; \frac{h}{k}\right) = ie^{\frac{\pi i h}{4k}} g_2(i^a e^{\frac{\pi i Ah}{Ck}}; e^{\frac{\pi i h}{k}}),
\]
and has a well-defined value in $\mathbb{C}$. 
\end{lemma}

\begin{proof}
First, we note that if $\alpha \not\in \frac12 \mathbb Z \tau + \frac12 \mathbb Z$, then using the \eqref{def:eta} and \eqref{eq:theta}, we can write
\begin{equation}\label{eq:muKang}
\mu\left(2\alpha, \frac{\tau}{2};\tau\right) = iq^\frac18 g_2(e^{2\pi i \alpha};q^\frac12) - iq^{\frac{1}{8}} \cdot \frac{(q;q)_\infty (-q^{\frac12};q^{\frac12})_\infty^2}{(e^{4\pi i \alpha};q)_\infty (e^{-4\pi i\alpha}q;q)_\infty}.
\end{equation}

Replacing $\tau$ with $\frac{h}{k}\in S$ in \eqref{eq:muKang} causes the $(q;q)_\infty$ term to vanish.  Moreover, we see that when setting
$\alpha = \frac{Ah}{2Ck} + \frac{a}{4}$, where $\frac{A}{2C}$, and $\frac{a}{4}$ are not both in $\frac12 \mathbb{Z}$, the terms $(e^{4\pi i \alpha};q)_\infty$ and $(e^{-4\pi i\alpha}q;q)_\infty$ will not vanish due to our hypotheses on $\widetilde{S}$.  Thus, we have that
\[
\mu\left(\frac{Ah}{Ck} + \frac{a}{2}, \frac{h}{2k}; \frac{h}{k}\right) = ie^{\frac{\pi i h}{4k}} g_2(i^a e^{\frac{\pi i Ah}{Ck}}; e^{\frac{\pi i h}{k}}),
\]
and it remains to show that this has a well-defined value in $\C$.

By definition, 
\begin{equation}\label{eq:g_2modified}
g_2(e^{2\pi i \alpha};q^\frac12) = \sum_{n=0}^\infty \frac{(-q^\frac12; q^\frac12)_n q^{n(n+1)/4}}{(e^{2\pi i \alpha};q)_{n+1}(e^{-2\pi i \alpha}q;q)_{n+1}}.
\end{equation}

Following similar reasoning as above, we note that replacing $\tau$ with $\frac{h}{k}\in S$ in \eqref{eq:g_2modified} causes the $(-q^\frac12; q^\frac12)_n$ term to vanish when $n\geq k$.  Moreover, we see that when setting
$\alpha = \frac{Ah}{2Ck} + \frac{a}{4}$, where $\frac{A}{2C}$, and $\frac{a}{4}$ are not both in $\frac12 \mathbb{Z}$, the terms $(e^{2\pi i \alpha};q)_{n+1}$ and $(e^{-2\pi i \alpha}q;q)_{n+1}$ will not vanish for any $n\geq 1$ due to our hypotheses on $\widetilde{S}$.  Thus we see that under our hypotheses, $g_2(i^a e^{\frac{\pi i Ah}{Ck}}; e^{\frac{\pi i h}{k}})$ is a finite sum, and thus $ie^{\frac{\pi i h}{4k}} g_2(i^a e^{\frac{\pi i Ah}{Ck}}; e^{\frac{\pi i h}{k}})$ gives a well-defined value in $\C$.
\end{proof}

\section{Mock Modularity}\label{mock}

In order to prove Theorem \ref{mockmodularity}, we need to show $V_{\alpha}$ can be completed to a harmonic Maass form, as defined in Definition \ref{MaassForm}. 

Throughout this section, let 
\begin{align}\label{uvtaudefs}
u_\tau &:= 2\alpha = \frac{A}{C} \tau + \frac{a}{2} \nonumber \\
v_\tau &:= \frac{\tau}{2},
\end{align}
so that 
\[
V_\alpha(\tau) = i^{a+1}q^{-\frac{(2A-C)^2}{8C^2}} \mu\left(u_\tau, v_\tau; \tau\right).
\]
In \eqref{muhat}, we saw the completion of the function $\mu(u,v;\tau)$, defined in Definition \ref{def:mu}, to $\widehat{\mu}(u,v;\tau)$.  Thus we define 
\begin{equation}\label{Vhatdef}
\widehat{V}_{\alpha}(\tau) :=i^{a+1}q^{-\frac{(2A-C)^2}{8C^2}}\widehat{\mu}\left(u_\tau, v_\tau; \tau \right),
\end{equation}
with the aim of showing that $\widehat{V}_{\alpha}(\tau)$ is a harmonic Maass form.

We define the groups $A_{\alpha}$ by
\begin{equation}\label{mockgp}
A_{\alpha} = 
\begin{cases}
\Gamma^1(\lcm(2,C))       & \text{ if } a\in \{0,2\}, \\
\Gamma^1(\lcm(2,C))\cap \Gamma_0(2) & \text{ if } a\in \{1,3\} \text{ and } C\not\equiv 0 \pmod{4}, \\
\left(\Gamma^1(C)\cap \Gamma_0(2)\right) \cup\\ \quad \left(\left(\Gamma^1(\frac{C}{2})\cap\Gamma^0(C)\right)\setminus\left(\Gamma_0(2)\cup\Gamma^1(C)\right)\right)   & \text{ if } a\in \{1,3\} \text{ and } C\equiv 0 \pmod{4}. \\
\end{cases}
\end{equation}

It is not immediately obvious that $A_\alpha$ is a group when $a\in \{1,3\}$ and $C\equiv 0 \pmod{4}$.  We prove this in the following lemma.  When $C\equiv 0 \pmod{4}$, we define the following sets for convenience,
\begin{align}\label{GammaDefs}
\G_1 &:= \G^1(C) \cap \G_0(2) \nonumber \\
\G_2 &:= \G^1\left( \frac{C}{2} \right) \cap \G^0(C) \\
\G_3 &:= \G_0(2) \cup \G^1(C). \nonumber
\end{align}

\begin{lemma}
Let $C\equiv 0 \pmod{4}$.  Then the set 
\[
G = \G_1 \cup (\G_2 \backslash \G_3)
\] 
is a congruence subgroup of $SL_2(\Z)$.
\end{lemma}

\begin{proof}
We first observe that the principal congruence subgroup $\G(\lcm(2,C))$ is a subgroup of $\G_1$ and thus of $G$.  We next observe that the elements of $\G_2 \backslash \G_3$ are those of the form 
\[
g = \begin{pmatrix} \frac{C}{2} k_1 + 1 & Ck_2 \\ 2k_3 + 1 & \frac{C}{2} k_4 + 1 \end{pmatrix},
\]
where $k_i\in \Z$, and at least one of $k_1,k_4$ is odd.  However, we must also have that $\det g=1$, so since
\[
\det g \equiv \frac{C}{2}(k_1+k_4) + 1 \pmod{C},
\]
it follows that both $k_1,k_4$ are odd.  Thus we have that $g\in \G_2 \backslash \G_3$ if and only if $g$ is of the form
\begin{equation}\label{gform}
g = \begin{pmatrix} Ck_1 + \frac{C}{2} + 1 & Ck_2 \\ 2k_3 + 1 & Ck_4 + \frac{C}{2} + 1 \end{pmatrix},
\end{equation}
for $k_i \in \Z$ such that $\det g=1$.  It follows from \eqref{gform} that $\G_2 \backslash \G_3$ is closed under taking inverses.  Since $\G_1$ is a group, it is clearly closed under inverses as well, and thus $G$ is closed under inverses.

We also need to show that $G$ is closed under multiplication.  Let $g,h\in G$. If $g,h\in \G_1$, then $gh\in \G_1$ since $\G_1$ is a group.  Suppose $g,h\in \G_2 \backslash \G_3$.  Then by \eqref{gform} we can write 
\begin{align*}
g &= \begin{pmatrix} Ck_1 + \frac{C}{2} + 1 & Ck_2 \\ 2k_3 + 1 & Ck_4 + \frac{C}{2} + 1 \end{pmatrix}, \\
h &= \begin{pmatrix} C\ell_1 + \frac{C}{2} + 1 & C\ell_2 \\ 2\ell_3 + 1 & C\ell_4 + \frac{C}{2} + 1 \end{pmatrix},
\end{align*}
for integers $k_i, \ell_i$ and so we have that 
\[
gh \equiv \begin{pmatrix} 1 & 0 \\ 2(k_3 + \ell_3 + 1) & 1 \end{pmatrix} \pmod{C},
\]
which shows that $gh \in \G_1 \subseteq G$.  Lastly, suppose that $g\in \G_1$ and $h\in \G_2 \backslash \G_3$.  Then by \eqref{gform},
\begin{align*}
g &= \begin{pmatrix} Ck_1 + 1 & Ck_2 \\ 2k_3 & Ck_4 + 1 \end{pmatrix}, \\
h &= \begin{pmatrix} C\ell_1 + \frac{C}{2} + 1 & C\ell_2 \\ 2\ell_3 + 1 & C\ell_4 + \frac{C}{2} + 1 \end{pmatrix},
\end{align*}
for integers $k_i, \ell_i$.  Thus, 
\[
gh \equiv  \begin{pmatrix} 1 & 0 \\ 2k_3 & 1 \end{pmatrix} \begin{pmatrix} \frac{C}{2} + 1 & 0 \\ 2\ell_3 + 1 & \frac{C}{2} + 1 \end{pmatrix} \equiv \begin{pmatrix} \frac{C}{2} + 1 & 0 \\ 2(k_3 + \ell_3) + 1 & \frac{C}{2} + 1 \end{pmatrix} \pmod{C},
\]
which shows that $gh\in \G_2 \backslash \G_3$.  Similarly, $hg \in \G_2 \backslash \G_3$ as well.  

\end{proof}

Our proof of Theorem \ref{mockmodularity} will utilize Lemma \ref{muhattrans}, so we next need the following lemma, which establishes that the groups we defined in \eqref{mockgp} are as large as possible for our method.  Recall that $\tilde{x}_{\gamma,\tau}$ is defined in \eqref{tildedef}.

\begin{lemma}\label{mockgpthm}
Fix $\alpha = \frac{A}{2C}\tau + \frac{a}{4}$, such that $A,a \in \Z$, $C\in \N$, $0 \leq a \leq 3$, and $\gcd(A,C)=1$, and let $\tau\in\mathcal{H}$. Then both $\tilde{u}_{\gamma,\tau} - u_{\tau}, \tilde{v}_{\gamma,\tau} - v_{\tau}\in \Z\tau +\Z$ if and only if $\gamma\in A_{\alpha}$.
\end{lemma}

\begin{proof}
We first observe that by Corollary \ref{congcoro}, $\tilde{v}_{\gamma,\tau} - v_{\tau}\in \Z\tau +\Z$ if and only if $\gamma\in \G^1(2)$.  So it remains to investigate $\tilde{u}_{\gamma,\tau} - u_{\tau}$.  

When $a$ is even, we can write
\[
u_\tau = \frac{A\tau + (a/2)C}{C},
\]
and thus by Corollary \ref{congcoro}, $\tilde{u}_{\gamma,\tau} - u_{\tau} \in \Z\tau +\Z$ if and only if $\gamma\in \G^1(C)$.  Together with our condition that $\tilde{v}_{\gamma,\tau} - v_{\tau}\in \Z\tau +\Z$ if and only if $\gamma\in \G^1(2)$, we see that both $\tilde{u}_{\gamma,\tau} - u_{\tau}, \tilde{v}_{\gamma,\tau} - v_{\tau}\in \Z\tau +\Z$ if and only if $\gamma\in \G^1(\lcm(2,C))$ as desired.

When $a$ is odd, we consider the cases when $C$ is odd, $C\equiv 2 \pmod{4}$, and $C\equiv 0 \pmod{4}$ separately.  

When $C$ is odd, we write
\[
u_\tau = \frac{2A\tau + aC}{2C},
\]
so by Lemma \ref{mockmodcong}, $\tilde{u}_{\gamma,\tau} - u_{\tau} \in \Z\tau +\Z$ if and only if $\gamma = \begin{pmatrix} x & y \\ z & w \end{pmatrix}$ satisfies
\begin{align}
2Ax +aCz &\equiv 2A \pmod{2C}  \label{aoddcong1}\\
2Ay + aCw &\equiv aC\pmod{2C}.  \label{aoddcong2}
\end{align}
We see that if $\gamma\in \G^1(2C) \cap \G_0(2) \subseteq \G^1(2)$, then \eqref{aoddcong1} and \eqref{aoddcong2} are clearly satisfied and thus both $\tilde{u}_{\gamma,\tau} - u_{\tau}, \tilde{v}_{\gamma,\tau} - v_{\tau}\in \Z\tau +\Z$.  Conversely, from \eqref{aoddcong1} we conclude that $z$ must be even and thus $\gamma \in \G_0(2)$.  Furthermore, since $\det\gamma =1$, it follows that $x,w$ must both be odd.  Since $w$ is odd, \eqref{aoddcong2} implies that $Ay\equiv 0 \pmod{C}$ and thus $C\mid y$ forcing $\gamma \in \G^0(C)$ as well.  Combining this with the fact that $\tilde{v}_{\gamma,\tau} - v_{\tau}\in \Z\tau +\Z$ implies $\gamma\in \G^1(2)$, we see that both $\tilde{u}_{\gamma,\tau} - u_{\tau}, \tilde{v}_{\gamma,\tau} - v_{\tau}\in \Z\tau +\Z$ implies that $\gamma\in \G_0(2) \cap \G^0(C) \cap \G^1(2)$.  Since we are assuming $C$ is odd, this group is simply $ \G^1(2C) \cap \G_0(2)$ as desired. 

When $C$ is even, we write
\[
u_\tau = \frac{A\tau + a(C/2)}{C},
\]
so by Lemma \ref{mockmodcong}, $\tilde{u}_{\gamma,\tau} - u_{\tau} \in \Z\tau +\Z$ if and only if $\gamma = \begin{pmatrix} x & y \\ z & w \end{pmatrix}$ satisfies
\begin{align}
Ax +a\left(\frac{C}{2}\right)z &\equiv A \pmod{C}  \label{aoddcong3}\\
Ay + a\left(\frac{C}{2}\right)w &\equiv a\left(\frac{C}{2}\right) \pmod{C}.  \label{aoddcong4}
\end{align}
We see that if $\gamma\in \G^1(C) \cap \G_0(2) \subseteq \G^1(2)$, then \eqref{aoddcong3} and \eqref{aoddcong4} are clearly satisfied.  Moreover, when $C\equiv 0 \pmod{4}$,  $\gamma\in \G_2 \backslash \G_3 \subseteq \G^1(2)$ implies that \eqref{aoddcong3} and \eqref{aoddcong4} are satisfied.  Thus for $\gamma \in A_\alpha$, both $\tilde{u}_{\gamma,\tau} - u_{\tau}, \tilde{v}_{\gamma,\tau} - v_{\tau}\in \Z\tau +\Z$.    Conversely, assuming both $\tilde{u}_{\gamma,\tau} - u_{\tau}, \tilde{v}_{\gamma,\tau} - v_{\tau}\in \Z\tau +\Z$ gives us immediately that $\gamma \in \G^1(2)$, so $y$ is even and both of $x,w$ are odd.  Since $w$ is odd, \eqref{aoddcong4} implies that $Ay\equiv 0 \pmod{C}$ and thus $C\mid y$ forcing $\gamma \in \G^0(C)$.

We now consider specifically when $C\equiv 2 \pmod{4}$.  Since $x$ and $\frac{C}{2}$ are odd, \eqref{aoddcong3} forces that $z$ must be even, and so $\gamma \in \G_0(2)$.   Moreover, $z$ even means that \eqref{aoddcong3} implies that $A(x-1) \equiv 0 \pmod{C}$ and thus $$x\equiv 1 \pmod {C}.$$  By the fact that $\det\gamma =1$, we thus have that $\gamma\in \G^1(C)$.  Combining this with the fact that $\tilde{v}_{\gamma,\tau} - v_{\tau}\in \Z\tau +\Z$ implies $\gamma\in \G^1(2)$, we see that both $\tilde{u}_{\gamma,\tau} - u_{\tau}, \tilde{v}_{\gamma,\tau} - v_{\tau}\in \Z\tau +\Z$ implies that $\gamma\in \G_0(2) \cap \G^1(C)$ as desired.

Now, suppose $C\equiv 0\pmod{4}$. Here, we note that $\frac{C}{2}$ is even so we no longer have any restrictions on the parity of $z$. If $z$ is even, then similarly to the case when $C\equiv 2\pmod{4}$, we know that $x,w \equiv 1\pmod{C}$ and both $\tilde{u}_{\gamma,\tau} - u_{\tau}, \tilde{v}_{\gamma,\tau} - v_{\tau}\in \Z\tau +\Z$ implies that $\gamma\in \G_0(2) \cap \G^1(C)$ as desired.  If $z$ is odd however, then taking the inverse of $A$ in $(\Z/C\Z)^\times$, we can rewrite (\ref{aoddcong3}) as 
\[
x\equiv 1+A^{-1}\frac{C}{2} \equiv 1+\frac{C}{2}\pmod{C}.
\] 
Thus $x\equiv 1 \pmod{\frac{C}{2}}$, and since $\det\gamma =1$, it follows that $w\equiv 1 \pmod{\frac{C}{2}}$ also.  So in this case we must have that $\gamma \in \G_2 \backslash \G_3 \subseteq \G^1(2)$, as defined in \eqref{GammaDefs}. Therefore, both $\tilde{u}_{\gamma,\tau} - u_{\tau}, \tilde{v}_{\gamma,\tau} - v_{\tau}\in \Z\tau +\Z$ implies that $\gamma \in \G_1 \cup (\G_2 \backslash \G_3)$ as desired.
\end{proof}

%Before giving our proof of Theorem \ref{mockmodularity} we need one additional technical lemma which shows that the following expression in $\gamma = \smmat{x}{y}{z}{w}\in A_{\alpha}$, and $\tau\in\H$, 
%\begin{equation*}
%\phi_{\gamma,\tau} := e\left(\frac{-(2A-C)^2}{8C^2}\gamma\tau\right) e\left( \frac{-y}{2(y\tau+z)}(\tilde{u}_{\gamma,\tau}-\tilde{v}_{\gamma,\tau})^2\right) e\left( (u_\tau-v_\tau)(k_\gamma-r_\gamma)\right) q^{\frac{1}{2} (k_\gamma - r_\gamma)^2} q^{\frac{(2A-C)^2}{8C^2}}, %CHECK SIGN FOR q TERM.
%\end{equation*}
%reduces to a root of unity.  

%\begin{lemma}\label{phiroot} For $\gamma = \smmat{w}{x}{y}{z}\in A_{\alpha}$ and $\tau\in\H$, $\phi_{\gamma,\tau}$ is a root of unity. In particular, \begin{align}\phi_{\gamma,\tau} = e\left( -\frac{(2A-C)^2}{8C^2}xy -\frac{a(2A-C)}{4C}x(w-1) -\frac{a^2}{8}z(w-2) \right).
%\end{align}
%\end{lemma}

%\begin{proof}
%Let $\gamma = \smmat{x}{y}{z}{w}\in A_{\alpha}$.  Then by 

%\end{proof}

We now prove Theorem \ref{mockmodularity}.

\begin{proof}[Proof of Theorem \ref{mockmodularity}]
Let $u_\tau:= 2\alpha = \frac{A}{C} \tau + \frac{a}{2}$, $v_\tau:= \tau/2$, so that $$V_\alpha(\tau) = i^{a+1}q^{-\frac{(2A-C)^2}{8C^2}} \mu\left(u_\tau, v_\tau; \tau\right).$$
To show $V_{\alpha}$ is a mock modular form, we need to show that its completion $\widehat{V}_{\alpha}(\tau)$, as given in \eqref{Vhatdef}, is a harmonic Maass form as defined in Definition \ref{MaassForm}.  

Let $\gamma = \smmat{x}{y}{z}{w}\in A_{\alpha}$, and $\tau\in\H$. By Lemma \ref{mockgpthm}, we can write $\tilde{u}_{\gamma,\tau} = u_\tau +k_\gamma \tau + \ell_\gamma$ and $\tilde{v}_{\gamma,\tau} = v_\tau + r_\gamma \tau + s_\gamma$, where $k_\gamma,\ell_\gamma,r_\gamma,s_\gamma\in\Z$. Thus, by \eqref{tildedef} and Lemma \ref{muhattrans}, we have that  
\begin{multline*}
\widehat{\mu}(u_{\gamma\tau},v_{\gamma\tau};\gamma\tau) 
= \psi(\gamma)^{-3}(-1)^{k_\gamma+\ell_\gamma+r_\gamma+s_\gamma} (z \tau + w)^{\frac12}\\
\cdot q^{\frac{(k_\gamma - r_\gamma)^2}{2}}  \cdot e\left(\frac{-z(\tilde{u}_{\gamma,\tau} - \tilde{v}_{\gamma,\tau})^2}{2(z\tau+w)} + (k_\gamma - r_\gamma)(u_\tau-v_\tau)\right)\widehat{\mu}(u_\tau,v_\tau;\tau),
\end{multline*} 
where $\psi(\gamma)$ is as defined in \eqref{psidef}. Thus, we can write 
\begin{equation}\label{Vhattransdef}
\widehat{V}_{\alpha}(\gamma\tau) = \left(\psi(\gamma)^{-3}(-1)^{k_\gamma + \ell_\gamma + r_\gamma + s_\gamma} \phi_{\gamma,\tau}\right)(z\tau + w)^{\frac12} \widehat{V}_{\alpha}(\tau),
\end{equation}
where 
\begin{equation}\label{phidef}
\phi_{\gamma,\tau} := e\left(\frac{-(2A-C)^2}{8C^2}\gamma\tau\right) e\left( \frac{-z}{2(z\tau+w)}(\tilde{u}_{\gamma,\tau}-\tilde{v}_{\gamma,\tau})^2\right) e\left( (u_\tau-v_\tau)(k_\gamma-r_\gamma)\right) q^{\frac{1}{2} (k_\gamma - r_\gamma)^2} q^{\frac{(2A-C)^2}{8C^2}}. %CHECK SIGN FOR q TERM.
\end{equation}

To see that $\phi_{\gamma,\tau}$ is in fact a root of unity, we need to simplify the expression in \eqref{phidef} and observe that all terms containing $\tau$ cancel.  First observe that $k_\gamma - r_\gamma$ can be determined as the coefficient of $\tau$ in $(\tilde{u}_{\gamma,\tau}-\tilde{v}_{\gamma,\tau}) - (u_\tau-v_\tau)$.  Thus,
\begin{equation}\label{k-r}
k_\gamma - r_\gamma = \frac{(2A-C)}{2C}(x-1) + \frac{a}{2}z.
\end{equation}
From \eqref{k-r} we see that 
\begin{multline*}
q^{\frac{1}{2} (k_\gamma - r_\gamma)^2}  q^{\frac{(2A-C)^2}{8C^2}} = \\
e\left( \frac{(2A-C)^2}{8C^2} x^2 \tau - \frac{(2A-C)^2}{4C^2} x\tau  + \frac{(2A-C)^2}{4C^2} \tau  + \frac{a(2A-C)}{4C} xz\tau - \frac{a(2A-C)}{4C} z\tau + \frac{a^2z^2}{8}\tau \right) 
\end{multline*}
and also that
\begin{multline*}
e\left( (u_\tau-v_\tau)(k_\gamma-r_\gamma)\right) = \\ 
e\left( \frac{(2A-C)^2}{4C^2} x\tau  - \frac{(2A-C)^2}{4C^2} \tau + \frac{a(2A-C)}{4C}\ z\tau \right) e\left(\frac{a(2A-C)}{4C} (x-1) + \frac{a^2}{4} z \right).
\end{multline*}
Thus together we obtain that
\begin{multline}\label{eq345}
e\left( (u_\tau-v_\tau)(k_\gamma-r_\gamma)\right)q^{\frac{1}{2} (k_\gamma - r_\gamma)^2}  q^{\frac{(2A-C)^2}{8C^2}} = \\
e\left( \frac{(2A-C)^2}{8C^2} x^2 \tau + \frac{a(2A-C)}{4C} xz\tau + \frac{a^2}{8}z^2\tau \right) e\left(\frac{a(2A-C)}{4C} (x-1) + \frac{a^2}{4} z \right).
\end{multline}
Moreover, the first two terms in \eqref{phidef} also simplify using the fact that $\det\gamma =1$, yielding that
\begin{multline}\label{eq12}
e\left(\frac{-(2A-C)^2}{8C^2}\gamma\tau\right) e\left( \frac{-z}{2(z\tau+w)}(\tilde{u}_{\gamma,\tau}-\tilde{v}_{\gamma,\tau})^2\right) = \\
e\left( -\frac{x(2A-C)^2}{8C^2}(x\tau+y) - \frac{az(2A-C)}{4C}(x\tau+y) - \frac{a^2z}{8}(z\tau+w)\right) = \\
e\left(-\frac{(2A-C)^2}{8C^2}x^2\tau - \frac{a(2A-C)}{4C} xz\tau - \frac{a^2}{8}z^2\tau \right) e\left( -\frac{(2A-C)^2}{8C^2}xy - \frac{a(2A-C)}{4C} yz - \frac{a^2}{8}zw \right).
\end{multline}
Thus by \eqref{phidef}, \eqref{eq345}, and \eqref{eq12}, and again using the fact that $\det\gamma =1$, we have that
\begin{align}
\phi_{\gamma,\tau} = e\left( -\frac{(2A-C)^2}{8C^2}xy -\frac{a(2A-C)}{4C}x(w-1) -\frac{a^2}{8}z(w-2) \right) ,
\end{align}
which is clearly a root of unity.

Since $\phi_{\gamma,\tau}$ is a root of unity, \eqref{Vhattransdef} shows that $\widehat{V}_{\alpha}(\tau)$ obeys a modular transformation law with a suitable multiplier.  

To study the growth of $\widehat{V}_{\alpha}(\tau)$ at the cusp $\tau= \frac{x}{z}$, where $x,z$ are integers with $\gcd(x,z)=1$, we let $\smmat{x}{y}{z}{w}\in \rm{SL}_2(\mathbb{Z})$ and study the growth of $\widehat{V}_{\alpha}(\gamma \tau)$ at infinity.   Using Lemma \ref{zwemutrans} (2), and simplifying significantly, we have that
\begin{multline*}
\widehat{V}_{\alpha}(\gamma \tau) =  i^{a+1} e\left(\frac{-(2A-C)^2}{8C^2} \gamma \tau \right) e\left(\frac{-z \left( (\frac{A}{C}- \frac12 )(x\tau+y) + \frac{a}{2}(z\tau+w) \right)^2}{2(z\tau+w)} \right) \psi(\gamma)^{-3} (z\tau+w)^{\frac12} \\
\cdot \widehat{\mu}\left( \frac{A}{C}(x\tau+y) + \frac{a}{2} (z\tau+w), \frac12(x\tau+y) ; \tau \right) \\
= i^{a+1}e\left( \frac{a(2A-C)}{8C} - \frac{(2A-C)^2}{8C^2} x(x\tau+y) - \frac{a(2A-C)}{8C} x(z\tau+w) - \frac{a^2}{8} z(z\tau+w) \right) \\
\cdot \psi(\gamma)^{-3}  q^{\frac{-(x(2A-C)+zaC)^2}{8C^2}} (z\tau+w)^{\frac12} \cdot \widehat{\mu}\left( \left( \frac{Ax}{C} + \frac{az}{2}\right) \tau + \left(\frac{Ay}{C} + \frac{aw}{2} \right) , \frac{x}{2}\tau + \frac{y}{2} ; \tau \right)\\
=i^{a+1}e\left( \frac{-(xy(2A-C)^2 + yzaC(2A-C) + zwa^2C^2)}{8C^2} \right) \psi(\gamma)^{-3} \\
\cdot q^{\frac{-(x(2A-C)+zaC)^2}{8C^2}} (z\tau+w)^{\frac12} \cdot \widehat{\mu}\left( \left( \frac{Ax}{C} + \frac{az}{2}\right) \tau + \left(\frac{Ay}{C} + \frac{aw}{2} \right) , \frac{x}{2}\tau + \frac{y}{2} ; \tau \right),
\end{multline*}
which we observe is a constant multiple of a function of the form in \eqref{hmfgrowth}. Thus  $\widehat{V}_{\alpha}(\tau)$ meets the prescribed growth conditions of a harmonic Maass form.

Lastly, we need to show that $\widehat{V}_{\alpha}(\tau)$ is annihilated by the Laplacian $\Delta_{\frac12}$. This follows directly from Lemma \ref{LaplaceLem} after observing that 
\[
\widehat{V}_{\alpha}(\tau) = - \frac{i^{a+1}}{\sqrt{2}} e\left(  \left( \frac{A}{C} - \frac12 \right) \left( \frac{a-1}{2} \right) \right) \widehat{M}_{(\frac{A}{C} - \frac12 ), -\frac{a}{2} }(\tau),
\]
where $\widehat{M}_{a,b}(\tau)$ is defined in \eqref{Mabdef}.  

\end{proof}

\section{Quantum Modularity}\label{quantum}

We call a subset $S\subseteq\mathbb{Q}$ a {\it quantum set} for a function $F$ with respect to the group $G\subseteq \rm{SL}_2(\mathbb{Z})$ if both $F(x)$ and $F(Mx)$ exist (are non-singular) for all $x\in S$ and $M\in G$.  We first establish quantum sets and corresponding groups for our functions $V_\alpha$ defined in Definition \ref{deffunction}.  We make the following definitions, recalling the definition of $S$ in \eqref{S_def},

\begin{align}\label{def:moreS}
 S_{C1} &= \left\{\frac{h}{k} \in S\, \bigg|\, C \nmid h \right\},  \nonumber \\
 S_{C2} &= \left\{\frac{h}{k} \in S\, \bigg|\, C \nmid 2h\right\}, \\
S_{ev} &= \left\{\frac{h}{k} \in S_{C1}\, \bigg|\, k \equiv 0  \!\!\!\! \pmod{2} \right\}. \nonumber
\end{align}

We are now able to define our quantum sets and corresponding groups for $V_\alpha$.
\begin{definition}\label{def:S_a}
Fix $\alpha = \frac{A}{2C}\tau + \frac{a}{4}$, such that $A,a \in \Z$, $C\in \N$, $0 \leq a \leq 3$, and $\gcd(A,C)=1$. We define
\begin{equation}\label{sets}
  S_{\alpha} = \left \{
  \begin{aligned}
    &S_{C1}, && \text{if}\ a=0,2 \\
    &S_{C2} \cup S_{ev}, && \text{if}\ a = 1,3 .
  \end{aligned} \right. 
\end{equation}
\end{definition}

\begin{definition}\label{def:G_a}
Fix $\alpha = \frac{A}{2C}\tau + \frac{a}{4}$, such that $A,a \in \Z$, $C\in \N$, $0 \leq a \leq 3$, and $\gcd(A,C)=1$. We define
\[
G_{\alpha} = 
\begin{cases}
    &\left< \begin{pmatrix} 1 & 0\\ 1 & 1 \end{pmatrix}, \begin{pmatrix} 1 & C\\ 0 & 1 \end{pmatrix}   \right>  \text{ if}\ a \text{ even and C even,}\\
    &\left< \begin{pmatrix} 1 & 0\\ 2 & 1 \end{pmatrix}, \begin{pmatrix} 1 & C\\ 0 & 1 \end{pmatrix}  \right>  \text{ if}\ a \text{ odd and C even,} \\
   &\left< \begin{pmatrix} 1 & 0\\ 1 & 1 \end{pmatrix}, \begin{pmatrix} 1 & 2C\\ 0 & 1 \end{pmatrix}  \right>  \text{ if}\ a \text{ even and C odd,} \\
   &\left< \begin{pmatrix} 1 & 0\\ 2 & 1 \end{pmatrix}, \begin{pmatrix} 1 & 2C\\ 0 & 1 \end{pmatrix}  \right>  \text{ if}\ a \text{ odd and C odd.}\\  
\end{cases}
\]

\end{definition}

Before we can prove Theorem \ref{quantummodularity} we must establish the following.

\begin{theorem}\label{setsandgroups}
Let $\alpha = \frac{A}{2C}\tau + \frac{a}{4}$, such that $A,a \in \Z$, $C\in \N$, $0 \leq a \leq 3$, and $\gcd(A,C)=1$.  Then the set $S_{\alpha}$ is a quantum set for $V_{\alpha}$ with respect to the group $G_{\alpha}$. In particular, for all $\tau \in S_{\alpha}$ and $M \in G_{\alpha}$, 
\[
V_{\alpha}(\tau) = i^{a+1}e\left(\frac{-(C-2A)^2}{8C^2}\tau\right)\mu\left(2\alpha, \frac{\tau}{2}; \tau\right)
\]
and
\[
V_{\alpha}(M\tau) = i^{a+1}e\left(\frac{-(C-2A)^2}{8C^2}M\tau\right)\mu\left(2\alpha, \frac{M\tau}{2}; M\tau\right)
\]
are well-defined. 
\end{theorem}

\begin{proof}

By Lemma \ref{lem:eta_vanishing}, in order to show that $V_{\alpha}(\tau)$ is well-defined for $\tau=\frac{h}{k} \in S_{\alpha}$, it suffices to show that for all $n\in \N$,
\[
 \frac{nh}{k} \pm \left( \frac{Ah}{Ck} + \frac{a}{2}\right) \not\in \mathbb{Z}.
\]
Equivalently, we show that for any $n\in \N$, $r \in \mathbb{Z}$, the following identities do not hold
\begin{align}\label{eqn:cont}
Ck(2r - a) = 2h(Cn + A)\\
Ck(2r + a) = 2h(Cn - A).
\end{align}

We first consider the case when $a=0$. Suppose $\frac{h}{k} \in S_{C1}$, and there exists an $n\in \N$ and $r \in \mathbb{Z}$ such that 
\begin{equation}\label{eqn:a=0}
Ckr = h(Cn \pm A).
\end{equation}
Then $C$ must divide the right-hand side of \eqref{eqn:a=0}, which gives a contradiction, since $C \nmid h$ by definition of $S_{C1}$, and $C\nmid A$ since they are relatively prime.

We next consider the case when $a = 2$. Suppose $\frac{h}{k} \in S_{C1}$, and there exists an $n\in \N$ and $r \in \mathbb{Z}$ such that 
\begin{equation}\label{eqn:a=2}
Ck(r \mp 1) = h(Cn \pm A).
\end{equation}
Then similarly $C$ must divide the right-hand side of \eqref{eqn:a=2}, which gives a contradiction, since $C \nmid h$ by definition of $S_{C1}$, and $C\nmid A$ since they are relatively prime.

We now consider the case when $a = 1,3$. Suppose $\frac{h}{k} \in S_{C2} \cup S_{ev}$, and there exists an $n\in \N$ and $r \in \mathbb{Z}$ such that
\begin{equation}\label{eqn:a=1,3}
Ck(2r \mp a) = 2h(Cn \pm A).
\end{equation}
Then again $C$ must divide the right-hand side of \eqref{eqn:a=1,3}, and we know $C\nmid A$ since they are relatively prime.  If $\frac{h}{k} \in S_{C2}$, then $C \nmid 2h$, which gives a contradiction. 
If $\frac{h}{k} \in S_{ev}$, then we have both that $C \nmid h$ and $k$ is even. Thus we have a contradiction since we can cancel the factor of $2$ on the right-hand side of \eqref{eqn:a=1,3} due to $k$ being even.

Since we have shown that $V_{\alpha}(\tau)$ is well-defined for $\tau=\frac{h}{k} \in S_{\alpha}$, to see that $V_{\alpha}(M\tau)$ is well-defined for $\tau=\frac{h}{k} \in S_{\alpha}$ and $M \in G_{\alpha}$, it suffices to show that if $\tau=\frac{h}{k} \in S_{\alpha}$ and $M$ is a generator or inverse of a generator of $G_{\alpha}$, then $M\tau\in S_\alpha$.

First, consider when $M_r = \left( \begin{smallmatrix} 1 & 0\\ r & 1 \end{smallmatrix} \right)$ for $r \in \{1,2\}$ is a generator of $G_\alpha$ as in Definition \ref{def:G_a}. Then, 
\[
M_r \frac{h}{k} = \frac{h}{rh + k}.
\]

We note that $h$ and $rh + k$ are relatively prime since $h$ and $k$ are relatively prime.
If $\tau=\frac{h}{k} \in S_{C1}$, then $h$ is odd and $C \nmid h$. Therefore, $M_r \tau \in S_{C1}$ as well. 

If $\frac{h}{k} \in S_{C2} \cup S_{ev}$, then by Definitions \ref{def:S_a} and \ref{def:G_a} we only need to consider when $r=2$.  If $\tau=\frac{h}{k} \in S_{C2}$, then $h$ is odd and $C \nmid 2h$. So, $M_r\tau \in S_{C2}$ as well. If $\tau=\frac{h}{k} \in S_{ev}$, then $k$ is even.  Since $r=2$ we have $rh + k$ is even. Thus if $\tau=\frac{h}{k} \in S_{C2} \cup S_{ev}$, then $M_r\tau \in S_{C2} \cup S_{ev}$ as well.
We can argue similarly to show that for $\tau\in S_{\alpha}$, $M_r^{-1} \tau \in S_{\alpha}$.

We now consider when $T_r = \left( \begin{smallmatrix} 1 & r\\ 0 & 1 \end{smallmatrix} \right)$  for $r \in \{C,2C\}$ is a generator of $G_\alpha$ as in Definition \ref{def:G_a}. Then, 
\[
T_r \frac{h}{k} = \frac{h + rk}{k}.
\]

Here also we have that $h + rk$ and $k$ are relatively prime since $h$ and $k$ are relatively prime.  If $\tau=\frac{h}{k} \in S_{C1}$, then we first note that $h+rk$ is again odd, since $h$ is odd, and $r$ must be even by Definitions \ref{def:S_a} and \ref{def:G_a}.  Moreover, $C\nmid h +rk$ since $C\nmid h$ and $C\mid r$.  Thus, $T_r\tau \in S_{C1}$.

If $\frac{h}{k} \in S_{C2} \cup S_{ev}$, then by Definitions \ref{def:S_a} and \ref{def:G_a} we know that $r$ is even.  If $\tau \in S_{ev}$, then $h + rk$ must be odd since $h$ is odd and $r$ is even. Also, $C\nmid h + rk$ since $C \nmid h$ and $C\mid r$. We also know $k$ is even. Therefore, $T_r\tau \in S_{ev}$.  If $\tau \in S_{C2}$, then $h + rk$ again remains odd since $h$ is odd and $r$ is even.  Because $C \nmid 2h$ and $C\mid r$, we must have $C \nmid 2(h+rk)$. Therefore, $T_r\tau \in S_{C2}$. Thus, if $\tau =\frac{h}{k}\in S_{C2} \cup S_{ev}$, then $T_r\tau \in S_{C2} \cup S_{ev}$.  We can argue similarly to show that for $\tau\in S_{\alpha}$, $T_r^{-1} \tau \in S_{\alpha}$.
\end{proof}

\begin{proposition}\label{ttransshift}
Let $\alpha = \frac{A}{2C}\tau + \frac{a}{4}$ where $a \in \{0,1,2,3\}$, $\mathrm{gcd}(A,C) = 1$, and $0 < \frac{A}{C} < 1$.  Additionally, let $T_r = \left( \begin{smallmatrix} 1 & r\\ 0 & 1 \end{smallmatrix} \right)$  for $r \in \{C,2C\}$ such that $T_r$ is a generator of $G_\alpha$ as in Definition \ref{def:G_a}.  Then for $\tau \in \mathbb{H}\cup S_{\alpha}$,
\[
V_\alpha(T_r\tau) = \zeta_8^{-r}(-1)^{\frac{Ar}{C}+\frac{r}{2}}e\left(-\frac{r}{2}\left(\frac{A}{C}-\frac{1}{2}\right)^2\right) V_\alpha(\tau).
\]
\end{proposition}

\begin{proof}
Observe that by applying $Ar/C$ iterations of Lemma \ref{lem_mu} (1), $r/2$ iterations of Lemma \ref{lem_mu} (2), and $r$ iterations of Lemma \ref{lem_mu} (4), we obtain 
\[
\mu\left(\left(\frac{A}{C} (\tau + r)  + \frac{a}{2}\right), \frac{\tau + r}{2}; \tau + r \right) = \zeta_8^{-r}(-1)^{\frac{Ar}{C}+\frac{r}{2}}\mu\left(\left(\frac{A}{C} \tau + \frac{a}{2}\right), \frac{\tau}{2}; \tau\right).
\]
By Definition \ref{deffunction} together with Theorem \ref{setsandgroups}, we obtain the result.
\end{proof}

We now consider the matrices $M_r = \left( \begin{smallmatrix} 1 & 0\\ r & 1 \end{smallmatrix} \right)$ for $r \in \{1,2\}$ which are generators of $G_\alpha$ as in Definition \ref{def:G_a}.  The corresponding transformations for $M_r$ are much more complicated than those for $T_s$.  The next lemma provides the first step in the process.  For convenience, we make the following definition which we will use throughout,

\begin{equation}\label{TauDefn}
\tau_r := - \frac{1}{\tau} -r.
\end{equation}

\begin{lemma}\label{mprops}
Let $\alpha = \frac{A}{2C}\tau + \frac{a}{4}$ where $a \in \{0,1,2,3\}$, $\mathrm{gcd}(A,C) = 1$, and $0 < \frac{A}{C} < 1$.  Additionally, let $M_r = \left( \begin{smallmatrix} 1 & 0\\ r & 1 \end{smallmatrix} \right)$ for $r\in \{1,2\}$ such that $M_r$ is a generator of $G_\alpha$ as in Definition \ref{def:G_a}.  Then for $\tau \in \mathbb{H}$,
\[
V_\alpha(M_r\tau) = (-1)^{\frac{ar}{2}} e\left(\frac{(a^2+1)r}{8}\right) \sqrt{r\tau+1} \cdot V_{\alpha}(\tau) + I_{\alpha,r}(\tau)+J_{\alpha,r}(\tau),
\]
where
\begin{align*}
I_{\alpha,r}(\tau)&:= \frac12 \, e\left(\frac{aA}{2C}\right) e\left(\frac{-a^2\tau_r}{8} \right) \sqrt{-i\tau_r} \, h\left(\frac{a}{2}\tau_r-\left(\frac{A}{C}-\frac{1}{2}\right); \tau_r\right),\\
J_{\alpha,r}(\tau)&:= -\frac12 i^a (-1)^{\frac{ar}{2}} e\left(\frac{(a^2+1)r}{8}\right) \sqrt{r\tau+1} q^{-\frac{1}{2}\left(\frac{1}{2}-\frac{A}{C}\right)^2} h\left(\left(\frac12 - \frac{A}{C} \right)\tau -\frac{a}{2}; \tau\right).
\end{align*}
\end{lemma}

\begin{proof}
Observe that by definition of $\tau_r$, $M_r \tau= S\tau_r = - \frac{1}{\tau_r}$.  Using Lemma \ref{lem_mu} (5) and the definition of $I_{\alpha,r}(\tau)$ above, we obtain
\begin{multline}\label{muStep1}
V_\alpha(M_r\tau) = i^{a+1} e\left(  \frac{1}{2\tau_r} \left( \frac{A}{C} - \frac12 \right)^2 \right) \cdot \mu\left( \frac{\frac{a}{2} \tau_r -\frac{A}{C}}{\tau_r} ,\frac{-\frac12}{\tau_r}; \frac{-1}{\tau_r}\right) \\
= -i^{a+1} e\left(\frac{a}{2} \left(\frac{A}{C} - \frac12 \right) \right) e\left(\frac{-a^2\tau_r}{8} \right)  \sqrt{-i\tau_r} \mu\left( \frac{a}{2} \tau_r - \frac{A}{C} , \frac{-1}{2}; \tau_r\right) + I_{\alpha,r}(\tau)\\
= -i^{a+1} e\left(\frac{a^2 r}{8} + \frac{a}{2} \left(\frac{A}{C} - \frac12 \right) \right) e\left(\frac{a^2}{8\tau} \right)  \sqrt{-i\tau_r} \mu\left( \frac{a}{2} \tau_r - \frac{A}{C} , \frac{-1}{2}; \tau_r\right) + I_{\alpha,r}(\tau).
\end{multline} 

Note that by Definition \ref{def:G_a}, we have that $\frac{ar}{2}$ is a positive integer.  By Lemma \ref{lem_mu} parts (1) and (4), we can write
\begin{equation}\label{muStep2}
\mu\left(\frac{a}{2} - \frac{A}{C}, \frac{-1}{2}; \tau_r  \right) = (-1)^{\frac{ar}{2}} e\left(\frac{r}{8}\right) \mu\left(  \frac{-a}{2\tau} - \frac{A}{C}, \frac{-1}{2} ; \frac{-1}{\tau}\right).
\end{equation}
Further applying Lemma \ref{lem_mu} (5) and Lemma \ref{lem_mu} (3) to \eqref{muStep2} and simplifying, we obtain 
\begin{multline*}
\mu\left(\frac{a}{2} - \frac{A}{C}, \frac{-1}{2}; \tau_r  \right) = (-1)^{\frac{ar}{2}} e\left(\frac{r}{8}\right) e\left( -\frac{a}{2} \left( \frac{A}{C} - \frac{a}{2} \right)\right) e\left(\frac{-a^2}{8\tau} \right) q^{-\frac12 \left( \frac{A}{C} - \frac12 \right)^2} \sqrt{-i\tau} \\
\cdot \left[ \frac{1}{2i} h \left(\left( \frac12 - \frac{A}{C} \right) \tau - \frac{a}{2} ; \tau \right) - \mu\left( \frac{A}{C} \tau + \frac{a}{2}, \frac{\tau}{2} ; \tau \right) \right].
\end{multline*}
Plugging this into \eqref{muStep1} and observing that $\sqrt{-i\tau_r}\cdot \sqrt{-i\tau} = \sqrt{r\tau + 1}$, we see that
\[
V_\alpha(M_r\tau) = (-1)^{\frac{ar}{2}} e\left(\frac{(a^2+1)r}{8}\right) \sqrt{r\tau+1} \cdot V_{\alpha}(\tau) + I_{\alpha,r}(\tau)+J_{\alpha,r}(\tau),
\]
as desired.
\end{proof}

In order to analyze $I_{\alpha,r}(\tau)+J_{\alpha,r}(\tau)$, we make the following definition. 

\begin{definition}\label{DeltaDef}
Let $x,y \in \C$, 
%be such that both $a$ and $b$ are not in $\Z+\frac{1}{2}$, 
and $\tau \in \mathbb{H}$; then
\[
\delta_{x,y}(\tau) := \int_0^{i\infty}\frac{g_{x+\frac{1}{2},y+\frac{1}{2}}(z)}{\sqrt{-i(z+\tau)}}\,dz +e\left(x\left(y+\frac{1}{2}\right)\right)q^{-\frac{x^2}{2}}h\left(x\tau-y;\tau\right).
\]
\end{definition}

\noindent With this definition, we are able to prove the following lemma, which is helpful in the proof of Theorem \ref{quantummodularity}.

\begin{lemma}\label{DeltaLem}
Let $x,y\in \mathbb{R}$ and $\tau\in\mathbb{H}$.  Then $\delta_{x,y}(\tau)$ satisfies
\begin{align*}
\delta_{x,y+1}(\tau) &= -e^{2\pi i x}\delta_{x,y}(\tau)+\frac{2e^{2\pi i x}}{\sqrt{-i\tau}}e^{\pi i (y+\frac{1}{2})^2/\tau},\\
\delta_{x+1,y}(\tau) &= \delta_{x,y}(\tau)-2e^{\pi i(x+y)}e^{2\pi i xy}q^{-(x+\frac{1}{2})^2/2}. 
\end{align*}
\end{lemma}

\begin{proof}
By Definition \ref{DeltaDef}, we have
$$
\delta_{x,y+1}(\tau) = \int_0^{i\infty}\frac{g_{x+\frac{1}{2},y+\frac{1}{2}+1}(z)}{\sqrt{-i(z+\tau)}}\,dz +e\left(x\left(y+\frac{1}{2}+1\right)\right)q^{-\frac{x^2}{2}}h\left(x\tau-y-1;\tau\right).
$$
Applying Lemma \ref{lem_gabprop} (2) and Lemma \ref{htrans} (1), we obtain
\begin{multline*}
\delta_{x,y+1}(\tau) =-e^{2\pi ix}\int_0^{i\infty}\frac{g_{x+\frac{1}{2},y+\frac{1}{2}}(z)}{\sqrt{-i(z+\tau)}}\,dz -e^{2\pi i x}e\left(x\left(y+\frac{1}{2}\right)\right)q^{-\frac{x^2}{2}}h\left(x\tau-y;\tau\right)\\
+\frac{2}{\sqrt{-i\tau}}e^{\pi i\left(x\tau-y-1/2\right)^2/\tau}e^{2\pi ix\left(y+\frac{1}{2}+1\right)}q^{-\frac{x^2}{2}}\\
=-e^{2\pi i x}\delta_{x,y}(\tau)+\frac{2e^{2\pi i x}}{\sqrt{-i\tau}}e^{\pi i (y+\frac{1}{2})^2/\tau}.
\end{multline*}
By Definition \ref{DeltaDef}, we similarly have
$$
\delta_{x+1,y}(\tau) = \int_0^{i\infty}\frac{g_{x+\frac{1}{2}+1,y+\frac{1}{2}}(z)}{\sqrt{-i(z+\tau)}}\,dz +e\left((x+1)\left(y+\frac{1}{2}\right)\right)q^{-\frac{(x+1)^2}{2}}h\left(x\tau-y+\tau;\tau\right).
$$
Applying Lemma \ref{lem_gabprop} (1) and Lemma \ref{htrans} (2), we have
\begin{multline*}
\delta_{x+1,y}(\tau) =\int_0^{i\infty}\frac{g_{x+\frac{1}{2},y+\frac{1}{2}}(z)}{\sqrt{-i(z+\tau)}}\,dz +e\left(x\left(y+\frac{1}{2}\right)\right)q^{-\frac{x^2}{2}}h\left(x\tau-y;\tau\right)\\
-2e^{\pi iy}e^{2\pi i x\left(y+\frac{1}{2}\right)}q^{-\frac{1}{2}\left(x^2+x+\frac{1}{4}\right)}\\
=\delta_{x,y}(\tau)-2e^{\pi i(x+y)}e^{2\pi i xy}q^{-(x+\frac{1}{2})^2/2}.
\end{multline*}
\end{proof}

The following theorem gives the transformation formula for $V_\alpha$ under the action of $M_r$.

\begin{theorem}\label{VTransf}
Let $\alpha = \frac{A}{2C}\tau + \frac{a}{4}$ where $a \in \{0,1,2,3\}$, $\mathrm{gcd}(A,C) = 1$, $0 < \frac{A}{C} < 1$, and $\frac{A}{C} \neq \frac12$ when $a$ is odd.  Additionally, let $M_r = \left( \begin{smallmatrix} 1 & 0\\ r & 1 \end{smallmatrix} \right)$ for $r\in \{1,2\}$ such that $M_r$ is a generator of $G_\alpha$ as in Definition \ref{def:G_a}.  Then for $\tau \in \mathbb{H} \cup S_\alpha \backslash \{ -\frac{1}{r} \}$,
\begin{multline*}
V_\alpha(M_r\tau) = (-1)^{\frac{ar}{2}} e\left(\frac{(a^2+1)r}{8}\right) \sqrt{r\tau+1} \cdot V_{\alpha}(\tau) \\
+ \frac{i}{2} \sqrt{r\tau + 1} \, e\left( \frac{A}{C} \left( \frac{a-1}{2} \right)\right) e\left(\frac{r}{8} (a+1)^2 \right)  \int_{\frac{1}{r}}^{i\infty} \frac{g_{\frac{A}{C}, \frac{1}{2}-\frac{a}{2}}(z)}{\sqrt{-i(z+\tau)}} \, dz.
\end{multline*} 
\end{theorem}

\begin{proof}
We first consider when $\tau \in \mathbb{H}$.  Let $I_{\alpha,r}(\tau)$ and $J_{\alpha,r}(\tau)$ be defined as in Lemma \ref{mprops}.  From Definition \ref{DeltaDef}, as well as Lemma \ref{lem_gabprop}, we can rewrite 

\begin{equation}\label{eq:J1}
J_{\alpha,r}(\tau) = \frac{i}{2} e\left( \frac{r}{8}(a+1)^2 \right) e\left(\frac{A}{C} \left( \frac{a+1}{2} \right) \right) \sqrt{r\tau + 1} \cdot \left[\delta_{\frac12 - \frac{A}{C} , \frac{a}{2}} (\tau) + e\left( \frac{-A}{C} \right)  \int_{0}^{i\infty}  \frac{g_{\frac{A}{C}, \frac12 -\frac{a}{2}}(z)}{\sqrt{-i(z+\tau)}} \, dz \right].
\end{equation}

We can similarly rewrite 
\begin{equation}\label{eq:I1}
I_{\alpha,r}(\tau) = -\frac12 \sqrt{-i\tau_r}  \int_{0}^{i\infty} \frac{g_{\frac{a}{2}+\frac{1}{2}, \frac{A}{C}}(z)}{\sqrt{-i(z+\tau_r)}} \, dz  + \frac12 \sqrt{-i\tau_r} \; \delta_{\frac{a}{2},\frac{A}{C} - \frac12}(\tau_r).
\end{equation}

Changing variables by letting $z=r-\frac{1}{u}$ in the integral
\[
\int_{0}^{i\infty} \frac{g_{\frac{a}{2}+\frac{1}{2}, \frac{A}{C}}(z)}{\sqrt{-i(z+\tau_r)}} \, dz,
\]
and recalling the definition \eqref{TauDefn}, we get that 
\[
\int_{0}^{i\infty} \frac{g_{\frac{a}{2}+\frac{1}{2}, \frac{A}{C}}(z)}{\sqrt{-i(z+\tau_r)}} \, dz = \sqrt{\tau} \int_{\frac{1}{r}}^{0} \frac{g_{\frac{a}{2}+\frac{1}{2}, \frac{A}{C}}(r-\frac{1}{u})}{\sqrt{i(u+\tau)}} \, \frac{du}{u^{3/2}}.
\]
Applying Lemma \ref{lem_gabprop} (1), (2), and (5), further yields that 
\begin{multline*}
 \int_{0}^{i\infty} \frac{g_{\frac{a}{2}+\frac{1}{2}, \frac{A}{C}}(z)}{\sqrt{-i(z+\tau_r)}} \, dz 
 = \sqrt{\tau} \, e\left(\frac{r}{8} (a+1)^2 \right)  \int_{\frac{1}{r}}^{0} \frac{g_{\frac{a}{2}+\frac{1}{2}, \frac{A}{C}}(-\frac{1}{u})}{\sqrt{i(u+\tau)}} \, \frac{du}{u^{3/2}}\\
 = -i \sqrt{-i \tau}\,e\left( \frac{A}{C} \left( \frac{a-1}{2} \right)\right) e\left(\frac{r}{8} (a+1)^2 \right)  \int_{\frac{1}{r}}^{0} \frac{g_{\frac{A}{C}, \frac{1}{2}-\frac{a}{2}}(u)}{\sqrt{-i(u+\tau)}} \, du,
\end{multline*}
so substituting this into \eqref{eq:I1} gives that
\begin{equation}\label{eq:I2}
I_{\alpha,r}(\tau) = \frac{i}{2} \sqrt{r\tau + 1} \, e\left( \frac{A}{C} \left( \frac{a-1}{2} \right)\right) e\left(\frac{r}{8} (a+1)^2 \right)  \int_{\frac{1}{r}}^{0} \frac{g_{\frac{A}{C}, \frac{1}{2}-\frac{a}{2}}(z)}{\sqrt{-i(z+\tau)}} \, dz
+ \frac12 \sqrt{-i\tau_r} \; \delta_{\frac{a}{2},\frac{A}{C} - \frac12}(\tau_r).
\end{equation}
Combining \eqref{eq:I2} and \eqref{eq:J1} gives that
\begin{multline*}
I_{\alpha,r}(\tau)+J_{\alpha,r}(\tau) = 
\frac{i}{2} \sqrt{r\tau + 1} \, e\left( \frac{A}{C} \left( \frac{a-1}{2} \right)\right) e\left(\frac{r}{8} (a+1)^2 \right)  \int_{\frac{1}{r}}^{i\infty} \frac{g_{\frac{A}{C}, \frac{1}{2}-\frac{a}{2}}(z)}{\sqrt{-i(z+\tau)}} \, dz \\
+ \frac12 \sqrt{-i\tau_r} \left[ \delta_{\frac{a}{2},\frac{A}{C} - \frac12}(\tau_r) +  
i \sqrt{-i \tau}  \, e\left( \frac{r}{8}(a+1)^2 \right) e\left(\frac{A}{C} \left( \frac{a+1}{2} \right) \right) \delta_{\frac12 - \frac{A}{C} , \frac{a}{2}} (\tau) \right]. 
\end{multline*} 
In order to establish the transformation for $\tau \in \mathbb{H}$, it thus remains to show that 
\begin{equation}\label{Delta0}
 \delta_{\frac{a}{2},\frac{A}{C} - \frac12}(\tau_r) +  
i \sqrt{-i \tau}  \, e\left( \frac{r}{8}(a+1)^2 \right) e\left(\frac{A}{C} \left( \frac{a+1}{2} \right) \right) \delta_{\frac12 - \frac{A}{C} , \frac{a}{2}} (\tau) = 0.
\end{equation}

By Theorem \ref{thm_z116} (2), we see that $\delta_{x,y} =0$ whenever $x,y\in (-\frac12, \frac12)$.  By hypothesis, $\frac{A}{C} \in (0,1)$, so in the case when $a=0$, we see trivially that \eqref{Delta0} vanishes.  When $a=2$,  we can similarly see that \eqref{Delta0} vanishes after using Lemma \ref{DeltaLem} to first shift each delta term into the case when $x,y\in (-\frac12, \frac12)$.  

When $a$ is odd, we first observe that $r=2$, by Definition \ref{def:G_a}.  We also assume that $\frac{A}{C} \neq \frac12$ so that we may use Lemma \ref{lem_Zlem} (2) and (3).  In the case when $a=1$, applying Lemma \ref{lem_Zlem} (2) shows that $\delta_{\frac{1}{2},\frac{A}{C} - \frac12}(\tau_r) =i$, and applying Lemma \ref{lem_Zlem} (3) shows that $\delta_{\frac12 - \frac{A}{C} , \frac{1}{2}} (\tau) = e(\frac12 - \frac{A}{C})/\sqrt{-i\tau}$, which together show that \eqref{Delta0} vanishes.  Similarly, in the case when $a=3$, we first use Lemma \ref{DeltaLem} to shift each delta term to match the case when $a=1$, which then shows that \eqref{Delta0} vanishes.

Now that we have established the transformation for $\tau \in \mathbb{H}$, continuation to $\tau \in S_\alpha \backslash \{ -\frac{1}{r} \}$ follows from Theorem \ref{setsandgroups}, and the fact that the integral appearing in the transformation is a real analytic function except at $-\frac{1}{r}$, as argued in \cite{BOPR, BR, FGKST, ZagierV}, for example.  
\end{proof} 

We can now easily prove Theorem \ref{quantummodularity}.

\begin{proof}[Proof of Theorem \ref{quantummodularity}]
Parts (1) and (2) follows as a corollary to Theorem \ref{VTransf} by considering the cases when $a=0,1,2,3$ separately, and applying Lemma \ref{lem_gabprop} (2) in the cases when $a=2,3$.  Part (3) follows as an immediate corollary to Proposition \ref{ttransshift} by considering the cases when $C$ is even or odd separately. 
\end{proof}

\end{document}